\documentclass[12pt]{amsart}
\usepackage{stmaryrd}
\usepackage{mathrsfs}
\usepackage{amsmath,amssymb}
\usepackage{amsfonts}
\usepackage{amsthm}
\usepackage{latexsym}
\usepackage{graphicx,color}
\def\R{\mathbb{R}}

\def\vv<#1>{\langle#1\rangle}

\def\1{\mathbf{1}}

\def\XXint#1#2{\setbox0=\hbox{$#1{#2}{\int}$}{#2}\kern-.5\wd0 }

\def\XXint#1#2#3{{\setbox0=\hbox{$#1{#2#3}{\int}$}
     \vcenter{\hbox{$#2#3$}}\kern-.5\wd0}}

\def\vv<#1>{{\left\langle#1\right\rangle}}

\def\wt{\widetilde}
\newtheorem{thm}{Theorem}[section]

\newtheorem{lem}{Lemma}[section]

\theoremstyle{definition}

\theoremstyle{remark}

\numberwithin{equation}{section}

\begin{document}
\title{An extension of Katsuda-Urakawa's Faber-Krahn inequality}

\author{Wankai He}
\address{Department of Mathematics, Shantou University, Shantou, Guangdong, 515063, China}
\email{18wkhe@stu.edu.cn}
\author{Chengjie Yu$^1$}
\address{Department of Mathematics, Shantou University, Shantou, Guangdong, 515063, China}
\email{cjyu@stu.edu.cn}
%\thanks{$^1$Research partially supported by GDNSF with contract no. 2026A1515012267. }
\thanks{$^1$Research partially supported by GDNSF with contract no. 2025A1515011144 and 2026A1515012267.}
\renewcommand{\subjclassname}{%
  \textup{2020} Mathematics Subject Classification}
\subjclass[2020]{Primary 05C35; Secondary 35R02}
\date{}
\keywords{Faber-Krahn inequality, Dirichlet eigenvalue, normalized combinatorial Laplacian }
\begin{abstract}
In this paper, motivated by our previous work \cite{HY}, we prove that the minimum of the first Dirichlet eigenvalues for the normalized combinatorial $p$-Laplacian  on connected finite graphs with boundary consisting of $n$ edges is only achieved by the tadpole graph $T_{n,3}$. This result extends the Faber-Krahn inequality of Katsuda-Urakawa \cite{KU} to normalized combinatorial $p$-Laplacians. Our argument is much simpler than that of Katsuda-Urakawa.
\end{abstract}
\maketitle
\markboth{He \& Yu}{Katsuda-Urakawa's Faber-Krahn inequality}
\section{Introduction}
In \cite{KU}, Katsuda and Urakawa obtained the following sharp Faber-Krahn inequality for the fist Dirichlet eigenvalues of the normalized combinatorial Laplacian on graphs with a fixed number of edges.
\begin{thm}[Katsuda-Urakawa \cite{KU}]\label{thm-KU}
Let $G$ be a connected graph with boundary that consists of $n\geq 4$ edges. Then,
$$\lambda_1(G)\geq \lambda_1(T_{n,3})$$
and the equality holds if and only if $G=T_{n,3}$.
\end{thm}
Here the pendant vertices of $G$ are considered as the boundary vertices of $G$, $T_{n,3}$ means the tadpole graph on $n$-vertices with the head a cycle of length $3$ (see the next section for details), and $\lambda_1(G)$ is the first Dirichlet eigenvalue of the normalized combinatorial Laplacian on $G$:
$$\Delta_Gf(x)=\frac{1}{\deg(x)}\sum_{y\sim x}(f(x)-f(y)).$$

The proof of Theorem \ref{thm-KU} in \cite{KU} is rather complicated using three kinds of surgeries to reduce the first Dirichlet eigenvalue of a graph with boundary. In this paper, motivated by our previous work \cite{HY}, we extend  Theorem \ref{thm-KU} to normalized combinatorial $p$-Laplacian:
$$\Delta_{p,G} f(x)=\frac{1}{\deg(x)}\sum_{y\sim x}|f(x)-f(y)|^{p-2}(f(x)-f(y)).$$
Our argument is much simpler than that of Katsuda-Urakawa \cite{KU} by just using a simple surgery on the graph.

\begin{thm}\label{thm-main}
Let $G$ be a connected graph with boundary that consists of $n\geq 4$ edges and $p>1$. Then,
$$\lambda_{1,p}(G)\geq \lambda_{1,p}(T_{n,3})$$
and the equality holds if and only if $G=T_{n,3}$.
\end{thm}

When $p=1$, the Faber-Krahn inequality in Theorem \ref{thm-main} still holds by that 
$$\lambda_{1,1}(G)=\lim_{p\to 1^{+1}}\lambda_{1,p}(G).\ \mbox{(see \cite{Ge} for example.)}$$
However, the rigidity part of the Faber-Krahn inequality does not hold. In fact, by noting that
$$\lambda_{1,1}(G)=h_D(G)$$
where $h_D(G)$ is the Dirichlet Cheeger constant (see \eqref{eq-Cheeger} for definition), one is not hard to see that 
$$\lambda_{1,1}(G)\geq\frac{1}{2n-1}$$
with equality if and only if $G$ has only one pendant vertex.

The classical Faber-Krahn inequality (see \cite[P. 87]{Ch}) says that 
$$\lambda_1(\Omega)\geq\lambda_1(B)$$
where $\Omega$ is any smooth bounded Euclidean domains and $B$ is a ball of the same volume as $\Omega$, and the equality holds if and only if $\Omega$ is the translation of the ball $B$. So, Theorem \ref{thm-KU} is a discrete version of Faber-Krahn inequality.

Discrete analogues of the classical Faber-Krahn inequality were first considered in the pioneer work \cite{Fr} of Friedman. In \cite{Fr}, Friedman formulated the problem of finding the Faber-Krahn inequality for domains in a homogeneous tree with fixed total length and conjectured that the minimum is achieved by a geodesic ball in the homogeneous tree. Later, Pruss \cite{Pr} disproved Friedman's conjecture, and Leydold \cite{Le97,Le02} completely solved Friedman's problem.  In \cite{Fr2}, Friedman established Faber-Krahn inequalities for eigenvalues of the combinatorial Laplacian on graphs with a fixed number of vertices, and the second named author and Yingtao Yu \cite{YY} extended the result to Steklov eigenvalues. The general problem of finding graphs satisfying the so called Faber-Krahn property in certain class of graphs was introduced by Bıyıkoğlu and Leydold \cite{BL}. In \cite{BL,ZZ12,ZZ13,ZZZ,WH,LLY}, the authors solved the extremum problem formulated in \cite{BL} for various  classes of graphs.
\section{Proof of the main result}
We first introduce some notations and basic facts.

For a nontrivial connected finite graph $G$, we denote the collection of pendant vertices of $G$ as $B(G)$ which is viewed as the boundary of $G$. The set
$$\Omega(G):=V(G)\setminus B(G)$$
is viewed as the interior of $G$. For $p>1$, the normalized combinatorial $p$-Laplacian operator on $G$ is defined as
$$\Delta_{p,G} f(x)=\frac{1}{\deg(x)}\sum_{y\sim x}|f(x)-f(y)|^{p-2}(f(x)-f(y)),\ \forall x\in V(G),$$
where $f\in \R^{V(G)}$. A real number $\lambda$ is called a $p$-Dirichlet eigenvalue of $G$ if the following Dirichlet boundary value problem:
$$\left\{\begin{array}{ll}\Delta_{p,G} f(x)=\lambda |f|^{p-2}f(x)&x\in \Omega(G)\\
f(x)=0&x\in B(G)
\end{array}\right.$$
has a nonzero solution $f$, and $f$ is called a $p$-Dirichlet eigenfunction of $G$. The smallest $p$-Dirichlet eigenvalue of $G$ is denoted as $\lambda_{1,p}(G)$ which is called the first $p$-Dirichlet eigenvalue of $G$, and the corresponding eigenfunction is called a first $p$-Dirichlet eigenfunction of $G$. The first $p$-Dirichlet eigenvalue can be characterized by the minimum of $p$-Rayleigh quotient:
\begin{equation}\label{eq-min}
\lambda_1(G)=\min_{f\in C_B(G)\setminus\{0\}}R_{p,G}[f],
\end{equation}
and the minimum is only achieved by first Dirichlet eigenfunctions. Here
$$C_B(G)=\left\{f\in \R^{V(G)}\ \Big|\ f|_B\equiv 0\right\},$$
and
$$R_{p,G}[f]=\frac{\|df\|_{p,G}^p}{\|f\|_{p,G}^p}$$
with
$$\|df\|_{p,G}^p=\sum_{\{x,y\}\in E(G)}|f(x)-f(y)|^p$$
and
$$\|f\|_{p,G}^p=\sum_{x\in V(G)}|f|^p(x)\deg(x)=\sum_{x\in\Omega(G)}|f|^p(x)\deg(x)$$
since $f\in C_B(G).$

If $B(G)\neq \emptyset$ and $p>1$, as shown in \cite{HW},
the first $p$-Dirichlet eigenvalue $\lambda_{1,p}(G)$ is positive and the first $p$-Dirichlet eigenfunction $f$ is nonzero and does not change signs in $\Omega(G)$. Without loss of generality, we can assume that $f$ is positive on $\Omega(G)$. In this case, we call $f$ a positive first $p$-Dirichlet eigenfunction of $G$. Moreover, as shown in \cite{HW}, $\lambda_1(G)$ is of multiplicity one in the sense that the first $p$-Dirichlet eigenfunction is unique up to a constant multiple.

When $p=1$, the definition of the normalized combinatorial $1$-Laplacian for a graph is subtle, see \cite{Ch} for example. However, the first $1$-Dirichlet eigenvalue $\lambda_{1,1}(G)$ can be also characterized by Rayleigh quotient in \eqref{eq-min} with $p=1$. In fact, it is well-known (see \cite{Gr,HW} for example) that 
\begin{equation}\label{eq-lambda1-Cheeger}
\lambda_{1,1}(G)=h_D(G)
\end{equation}
where 
\begin{equation}\label{eq-Cheeger}
h_D(G)=\inf_{\emptyset\neq U\subset\Omega(G)}\frac{|E(U,U^c)|}{\sum_{x\in U}\deg(x)}.
\end{equation}
By setting $f=\1_U$ with $U$ a nonempty subset in $\Omega(G)$ in \eqref{eq-min}, one has
\begin{equation*}
\lambda_{1,p}(G)\leq h_D(G).
\end{equation*} 
Moreover, by setting $U=\{x\}$ where $x$ is an interior vertex of $G$, one has
\begin{equation*}
h_D(G)\leq 1
\end{equation*}
and
\begin{equation}\label{eq-lambda1}
\lambda_{1,p}(G)\leq 1.
\end{equation} 

Next, we introduce the notion of tadpole graph. For $n>i\geq 3$, we denote the tadpole graph on $n$ vertices with the head a cycle of length $i$ as $T_{n,i}$ which can represented as
\begin{equation}\label{eq-tadpole-graph}
T_{n,i}: t_n\sim t_{n-1}\sim \cdots\sim t_i\sim t_{i-1}\sim \cdots \sim t_2\sim t_1\sim t_i.
\end{equation}
The path
$$P:t_n\sim t_{n-1}\sim \cdots\sim t_i$$
is called the tail of $T_{n,i}.$ The cycle
$$C:t_i\sim t_{i-1}\sim \cdots \sim t_2\sim t_1\sim t_i$$
is called the head of $T_{n,i}$. The vertices $t_i$ and $t_n $ are called the neck vertex and end vertex of $T_{n,i}$ respectively.  

Tadpole graphs have the following useful spectral properties with their proofs are the same as those in \cite{HY}. See also \cite{KU}.
\begin{lem}\label{lem-max-in-head}
 Let $n>i\geq 3$, $p>1$ and $f$ be a positive first $p$-Dirichlet eigenfunction of $T_{n,i}$. Then, $f$ does not achieve its maximum on the tail of $T_{n,i}$.
\end{lem}
\begin{lem}\label{lem-comparison-tadpole}
For $n\geq 5$ and $p>1$,
$$\lambda_{1,p}(T_{n,4})>\lambda_{1,p}(T_{n,3}).$$
\end{lem}
\begin{lem}\label{lem-comparison-path}
For any $n\geq 4$ and $p>1$,
$$\lambda_{1,p}(P_n)>\lambda_{1,p}(P_{n+1})>\lambda_{1,p}(T_{n,3}).$$
Here $P_n$ is the path graph on $n$ vertices.
\end{lem}
We are now ready to prove  Theorem \ref{thm-main}.
\begin{proof}
Let $f$ be a positive first $p$-Dirichlet function of $G$ and $m\in V(G)$ be a maximum point of $f$. Let
$$P: v_n\sim v_{n-1}\sim \cdots \sim v_i=m$$
be a shortest path joining a boundary vertex $v_n$ to $m$. Then,
\begin{equation*}
\begin{split}
2n=&\sum_{x\in B(G)}\deg(x)+\sum_{k=i+1}^{n-1}\deg(v_k)+\sum_{x\in \Omega(G)\setminus\{v_{i+1},\cdots,v_{n-1}\}}\deg(x)\\
=& |B(G)|+2(n-i-1)+\sum_{k=i+1}^{n-1}(\deg(v_k)-2)+\sum_{x\in \Omega(G)\setminus\{v_{i+1},\cdots,v_{n-1}\}}\deg(x).
\end{split}
\end{equation*}
So,
\begin{equation}\label{eq-degree}
\begin{split}
\sum_{k=i+1}^{n-1}(\deg(v_k)-2)+\sum_{x\in \Omega(G)\setminus\{v_{i+1},\cdots,v_{n-1}\}}\deg(x)=2(i+1)-|B(G)|\leq 2i+1.
\end{split}
\end{equation}

When $|E(G)|-|E(P)|\geq 3$, we have $i\geq 3$. Let
$$T_{n,3}:u_n\sim u_{n-1}\sim\cdots \sim u_i\sim\cdots\sim u_{3}\sim u_{2}\sim u_1\sim u_3$$
and
$$\wt f(u_k)=\left\{\begin{array}{ll}f(v_k)&i\leq k\leq n\\
f(v_i)&1\leq k<i.
\end{array}\right.$$
Then, by \eqref{eq-degree},
\begin{equation*}
\begin{split}
&\|df\|_{p,G}^p\\
=&\sum_{k=i+1}^{n-1}f^p(v_k)\deg(v_k)+\sum_{x\in \Omega(G)\setminus\{v_{i+1},\cdots,v_{n-1}\}}f^p(x)\deg(x)\\
\leq&2\sum_{k=i+1}^{n-1}f^p(v_k)+f^p(m)\left(\sum_{k=i+1}^{n-1}(\deg(v_k)-2)+\sum_{x\in \Omega(G)\setminus\{v_{i+1},\cdots,v_{n-1}\}}\deg(x)\right)\\
\leq&2\sum_{k=i+1}^{n-1}f^p(v_k)+(2i+1)f^p(m)\\
=&\|\wt f\|_{p,T_{n,3}}^p
\end{split}
\end{equation*}
and
\begin{equation*}
\|df\|_{p,G}^p\geq \sum_{k=i}^{n-1}|f(v_{k+1})-f(v_k)|^p=\sum_{k=i}^{n-1}|\wt f(u_{k+1})-\wt f(u_k)|^p=\|d \wt f\|_{p,T_{n,3}}^p.
\end{equation*}
Hence, by \eqref{eq-min},
\begin{equation}
\lambda_{1,p}(G)=R_{p,G}[f]\geq R_{p,T_{n,3}}[\wt f]> \lambda_{1,p}(T_{n,3}).
\end{equation}
The last inequality is strict because $\wt f$ is not a positive first $p$-Dirichlet eigenfunction of $T_{n,3}$ by Lemma \ref{lem-max-in-head}.

When $|E(G)|-|E(P)|=2$, we have $i=2$. Because $v_2$ is not a boundary vertex and $P$ is a shortest path joining $v_n$ and $v_2$, there is another vertex $v_1$ adjacent to $v_2$. If there is no other vertex of $G$, then the remaining edge of $G$ should be $\{v_1\sim v_j\}$ for some $j=3,4,\cdots, n-1$. Because $P$ is a shortest path joining $v_n$ and $v_2$, we know that $j=3$ or $j=4$.  When $j=3$, $G=T_{n,3}$ and we are done. When $j=4$, $G=T_{n,4}$. By Lemma \ref{lem-comparison-tadpole},
$$\lambda_{1,p}(G)=\lambda_{1,p}(T_{n,4})> \lambda_{1,p}(T_{n,3}).$$
Otherwise, let $v_0$ be another vertex of $G$ which should be one of the end points of the remaining edge. Suppose $v_0$ is adjacent to $v_i$ for some $i=1,2,\cdots,n-1$. When $v_0\sim v_1$, $G=P_{n+1}$. By Lemma \ref{lem-comparison-path},
$$\lambda_{1,p}(G)=\lambda_{1,p}(P_{n+1})>\lambda_1(T_{n,3}).$$
When $v_0\sim v_j$ for some $j=2,3,\cdots,n-1$. Let $G'=G-\{v_0\}$. Then, $G'=P_n$.
Moreover
$$\|df\|_{p,G'}^p= \|df\|_{p,G}^p-f^p(v_j)$$
and
$$\|f\|_{p,G'}^p=\|f\|_{p,G}^p-f^p(v_j).$$
By \eqref{eq-lambda1},
$$\frac{\|df\|_{p,G}^p}{\|f\|_{p,G}^p}=\lambda_{1,p}(G)\leq 1,$$
we have
$$R_{p,G'}[f]=\frac{\|df\|_{p,G}^p-f^p(v_j)}{\|f\|_{p,G}^p-f^p(v_j)}\leq \lambda_{1,p}(G).$$
So, by Lemma \ref{eq-min} and Lemma \ref{lem-comparison-path},
$$\lambda_{1,p}(G')=R_{p,G}[f]\geq R_{p,G'}[\wt f]\geq\lambda_{1,p}(P_n)>\lambda_{1,p}({T_{n,3}}).$$

Finally, when $|E(G)|-|E(P)|=1$, it is clear that $G=P_{n+1}.$ Then, by Lemma \ref{lem-comparison-path},
$$\lambda_{1,p}(P_{n+1})>\lambda_{1,p}(T_{n,3}).$$
This completes the proof of the Theorem.
\end{proof}

\end{document}